\documentclass[a4paper,12pt,reqno]{amsart}
\usepackage{amsfonts,amssymb,hyperref,amsthm,enumerate,
color,stmaryrd,pgf,tikz,comment,multirow,mathtools}
\usepackage[left=2.6cm,right=2.6cm,top=3cm,bottom=3cm,bindingoffset=0cm]
{geometry}

\usepackage{array}   \usepackage{multirow}  \usepackage{booktabs,caption,fixltx2e}
\usepackage[flushleft]{threeparttable}
\newtheorem{theorem}{Theorem}
\newtheorem{lemma}[theorem]{Lemma}
\newtheorem{proposition}[theorem]{Proposition}
\newtheorem{corollary}[theorem]{Corollary}
\theoremstyle{definition}

\newtheorem{problem}{Problem}

\DeclareMathOperator{\Aut}{Aut}

\DeclareMathOperator{\Out}{Out}
\DeclareMathOperator{\Cay}{Cay}
\DeclareMathOperator{\CM}{CM}
\DeclareMathOperator{\AGL}{AGL}
\DeclareMathOperator{\PSL}{PSL}
\DeclareMathOperator{\PGL}{PGL}
\DeclareMathOperator{\GL}{GL}
\DeclareMathOperator{\PGAL}{P\Gamma L}
\DeclareMathOperator{\Sym}{S}
\DeclareMathOperator{\ZZ}{\mathbb{Z}}
\DeclareMathOperator{\CC}{C}
\DeclareMathOperator{\DD}{D}

\DeclareMathOperator{\Alt}{A}
\DeclareMathOperator{\Mat}{M}
\DeclareMathOperator{\ch}{~char~}

   \def\lg{\langle} \def\rg{\rangle}
 
\def\og{\overline G} \def\oh{\overline H} \def\ok{\overline K}  \def\ox{\overline X}

\newcommand{\red}[1]{{\color[rgb]{1,0,0}{#1}}}

\title[]{On exact products of two dihedral groups}
\author[K.~Hu\and H. Yu]
{ Kan Hu\and Hao Yu$^{*}$}
\address{K. Hu,
\newline\indent
Department of Mathematics, Zhejiang Ocean University, Zhoushan, Zhejiang 316022, P.R. China
}
\email{hukan@zjou.edu.cn}

\address{H. Yu
\newline\indent
School of Mathematical Sciences,Capital Normal University, Beijing 100048, P.R. China
}
\email{3485676673@qq.com}
\thanks{{This research was supported by National Natural Science Foundation of China (11801507, 12071312).}
\newline\indent
$^{\ast}$ Corresponding author e-mail: 3485676673@qq.com
}
\begin{document}
\maketitle
\begin{abstract}
An exact product of two finite groups $H$ and $K$ is a finite group $X$ which contains $H$ and
$K$ as subgroups, satisfying $X=HK$ and $H\cap K=\{1_X\}$. In this paper,
we provide a classification of the exact products of two dihedral groups
of orders $2m$ and $2n$  for all odd numbers $m,n\geq 3$.
\end{abstract}
\section{Introduction}
Throughout the paper, groups considered are all finite.
A group $X$ is \textit{factorizable} if $X$ contains two subgroups $H$ and $K$
such that $X=HK$. In this case, we say that $H$ and $K$ furnish a \textit{factorization}.
The factorization $X=HK$ is \textit{proper} if both $H$ and $K$ are non-trivial subgroups of $X$,
and \textit{maximal} if both $H$ and $K$ are maximal subgroups of $X$.
Moreover, the factorization is \textit{exact}
if $H\cap K=\{1_X\}$, and in this case we also
say that $X$ is an \textit{exact product} of $H$ and $K$.
The following factorization problem arises naturally:
For a given group $X$, determine whether or not $X$ admits a proper  factorization,
and if so, find all proper factorizations of $X$; in particular, identify all maximal or exact factorizations of $X$.

In this paper we consider the factorization problem formulated in a slightly different way:
\begin{problem}\label{FP}
For two specified finite groups $H$ and $K$, describe and classify  the products, or more specifically,
the exact products $X=HK$ of $H$ and $K$.
\end{problem}
 Several famous results on this problem were obtained by some pioneering group theorists
in the 20th century. It\^o established that products of two abelian groups are metabelian~\cite{Ito1955},
 Douglas proved that products of two cyclic groups are supersolvable~\cite{Douglas1961}.
 Meanwhile Wielandt and Kegal~\cite{Kegel1961,Wielandt1951} showed that products of two nilpotent groups are
 necessarily solvable.

 Recent interest in exact products of two finite groups has focused on primitive
 permutation groups containing a regular subgroup and factorizations of finite almost simple
 groups~\cite{LPS2010,LWX2023,LPS1996, Xia2017}.
Hering, Liebeck and Saxl classified factorizations of exceptional groups of Lie type~\cite{HLS1987},
and the landmark work by Liebeck, Praeger, and Saxl classified maximal factorizations of almost simple groups~\cite{LPS1996}. Recently, all exact factorizations of the almost simple groups have been
classified by Li, Wang and Xia~\cite{LWX2023}.

A group is \textit{bicyclic} if it is a product of two cyclic subgroups. One of the longstanding
unsolved problems in this field is the classification of finite bicyclic groups. Besides the
aforementioned classical result by Douglas on supersolvability of bicyclic groups,
Huppert established that every bicyclic $p$-group  is metacyclic if $p>2$ is an odd prime~\cite{Huppert1967}.
However, it is well known that non-metacyclic bicyclic $2$-groups exist. The metacyclic $p$-groups
 were classified by Xu and Newman  for odd prime $p$~\cite{NX1988}, and  by Xu and Zhang
 for $p=2$. The bicyclic $2$-groups were eventually classified by Janko~\cite{Janko2008}.
Some results directly related to bicyclic groups
can be found in the context of regular embeddings of complete bipartite
graphs into orientable surfaces~\cite{Jones2010} , and in the context of Hopf algebras
arising from exact factorizations of bicyclic groups~\cite{ABM2014,ACIM2009}.
To the best of our knowledge it remains a challenging open
problem in the general case.

The interest in exact factorization with a cyclic factor extends to combinatorics.
A group $X$ is called a \textit{cyclic complementary extension} of $H$ if $X$ is an exact product of $H$
and a cyclic complement $K=\lg c\rg$; $X$ is also said to admit a \textit{cyclic complementary factorization}.
Cyclic complementary extensions of a given group $H$ have been intensively investigated
in the context of regular Cayley maps, and more generally, within the theory of skew morphisms.
Without going into the details, a \textit{Cayley map} $M:=\CM(H,S,\rho)$ is an embedding
of an (undirected, simple, and connected) Cayley graph $\Gamma:=\Cay(H,S)$ on
 a finite group $H$ into an oriented surface, where the  local rotation $R$ at each vertex
 is defined uniformly by setting
 \[
 R(h,hs)=R(h,h\rho(s)), h\in H, s\in S,
 \]
 where $S$ is an inverse-closed generating set of $H$ with $1_H\notin S$, and $\rho$
 is a full cycle on $S$. It is well known that every Cayley map contains a vertex-regular subgroup of
 automorphisms.  In the case where it is also arc-regular, it is termed a \textit{regular Cayley map}.
 It turns out that the automorphism group $X:=\Aut(M)$ of a regular Cayley map $M$ has
 a cyclic complementary factorization $X=H_LX_{1_H}$, where $H_L$ is the left regular representation
 of $H$, and $X_{1_H}$ is the (cyclic) stabilizer of the vertex $1_H$ of the map.
 Therefore, if we identify $H$ with $H_L$, the group $X$ can be regarded as a cyclic complementary extension
 of $H$.

 Every cyclic complementary extension $X=HK$ of a finite group $H$ by a cyclic complement $K=\lg c\rg\cong\CC_n$
gives rise to a permutation $\varphi$ on $H$, called a \textit{skew morphism}  of $H$,
 and an integer-valued function $\Pi:H\to\ZZ_n$ on $H$, called the \textit{extended power function}
 of $\varphi$,  which satisfy the following defining conditions:
 \begin{enumerate}[\rm(a)]
\item $\varphi(1_H)=1_H,$ and $\varphi(gh)=\varphi(g)\varphi^{\Pi(g)}(h)$ for all $g,h\in H$;
\item  $ \Pi(1_H)=1_H,$  and $ \Pi(gh)=\sum_{i=1}^{\Pi(g)}\Pi(\varphi^{i-1}(h))\pmod{n}$
for all $g,h\in H$.
 \end{enumerate}
 Remarkably, the cyclic complementary extension of $H$ by $K$ can be recovered from
 $\varphi$ and $\Pi$ in a canonical way~\cite{HJ2023}. In particular, a Cayley map $\CM(H,S,\rho)$
 is regular if and only if $\rho$ extends to a skew morphism of $H$~\cite[Theorem 1]{JS2002}.
  In this context  the cyclic complementary extensions of various families of groups
 have been investigated, including cyclic groups~\cite{CJT2016,CT2014,  HKK2024},
 dihedral groups~\cite{HKK2022, KK2021, ZD2016},  simple groups~\cite{BCV2022} and characteristically simple
 groups~\cite{CDL2022,DLYZ2023, DLY2023, DYL2023}, among others.

The main concern of the current paper is exact products of two dihedral groups,
which have important applications in unoriented regular Cayley maps~\cite{KK2006, HKK2024, Yu2024}.
Our main result is the following theorem, which determines the exact products of
two non-abelian dihedral groups $H\cong\DD_{2n}$ and $K\cong\DD_{2m}$ for all odd numbers $m,n\geq 3$

\begin{theorem}\label{main}
If $n,m\geq 3$ are odd numbers, then every exact product $X=HK$ of two
 dihedral groups $H=\lg x\rg\rtimes\lg y\rg\cong\DD_{2n}$ and
$K=\lg z\rg\rtimes\lg w\rg\cong\DD_{2m}$ has a presentation
\begin{equation}\label{Pre}
\begin{aligned}
X=\langle x,y,z,w|&x^n=y^2=z^m=w^2=[x,z]=1, x^y=x^{-1}, z^{w}=z^{-1},   \\
&[x,w]=x^{sn_1}z^{bm_1},[z,y]=x^{rn_1}z^{am_1},   [y,w]=x^{tn_1}z^{cm_1}\rangle,
\end{aligned}
\end{equation}
where $m_1$ and $n_1$ are positive divisors of $m$ and $n$ such that $\lg z\rg_X=\lg z^{m_1}\rg$ and
$\lg x\rg_X=\lg x^{n_1}\rg$,  the integers $a,b,c\in\mathbb{Z}_{m/m_1}$
and $r,s,t\in\mathbb{Z}_{n/n_1}$ satisfy the following conditions:
\begin{enumerate}[\rm(a)]
  \item $a(2+am_1)\equiv b(2+am_1)\equiv c(2+am_1)\equiv0\pmod{m/m_1},$
  \item $r(2+sn_1)\equiv s(2+sn_1)\equiv t(2+sn_1)\equiv0\pmod{n/n_1},$
  \item \red{$bk\equiv0\pmod{m/m_1}\iff k\equiv0\pmod{n_1},$ for any integer $k$,}
  \item \red{$rk\equiv0\pmod{n/n_1}\iff k\equiv0\pmod{m_1},$  for any integer $k$.}
\end{enumerate}
\end{theorem}
The paper is organized as follows. In Section~\ref{sec:pre} we gather all necessary notations, terminology, and known results from group theory for future reference. Section~\ref{sec:lems} presents important
characterizations on exact products of two dihedral groups, while the last Section~\ref{sec:class}  is dedicated to
proving Theorem~\ref{main}.

\section{Preliminaries}\label{sec:pre}
Throughout the paper, we denote the ring of residue classes of integers modulo a positive integer $n$ by $\ZZ_n$.
We also use $\CC_n,\DD_{2n}, \Alt_n$ and $\Sym_n$ to represent the cyclic group of order $n$, the dihedral group of order $2n$, the alternating group, and the symmetric group of degree $n$, respectively.
Additionally, we employ $\GL(n,q)$, $\PGL(n,q)$, $\PSL(n,q)$, and $\AGL(n,q)$ to denote the general linear group, the projective linear group, the projective special linear group, and the affine group, respectively,
 of dimension $n$ over a finite field $\mathbb{F}_q$ with $q:=p^e$ elements.

If $G$ is a group and $H$ is a subset of $G$, we denote the normalizer and centralizer of $H$ in $G$ by
 $N_G(H)$ and $C_G(H)$, respectively.
The following $N/C$ theorem is well known.
\begin{proposition}[\cite{Huppert1967}]
If $H$ is a subgroup of a group $G$, then $N_G(H)/C_G(H)$ is isomorphic to a subgroup of
$\Aut(H)$.
\end{proposition}

A finite group $G$ is \textit{solvable} if it contains a finite sequence of (subnormal) subgroups
\[
1=N_0\leq N_1\leq\cdots\leq N_s=G
\]
 such that $N_{i-1}\lhd N_{i}$ and $N_i/N_{i-1}$ is cyclic
of prime order for all $1\leq i\leq s$. In general, if $N\lhd G$ is a normal subgroup
 of a group $G$ with $F:=G/N$, then we say that $G$
is an \textit{extension} of $N$ by $F$. In particular, if the quotient group $F$ is cyclic, then
the extension is \textit{cyclic}. The following theorem on cyclic extensions is important for solvable groups.

\begin{proposition}[\protect{\cite[Theorem 3.36]{Isaacs2008}}]\label{Cext}
Let $m$ be a positive integer, and $N$ a group with $a\in N$ and $\sigma\in\Aut(N)$.
If $a^{\sigma}=a$ and $x^{\sigma^m}=x^a$ for all $x\in N$, then there exists a group $G$, unique up
to isomorphism, and having $N$ as a normal subgroup with the following properties:
\begin{enumerate}[\rm(a)]
\item $G/N=\langle gN\rangle$ is cyclic of order $m$;
\item $g^m=a$;
\item $x^{\sigma}=x^g$.
\end{enumerate}
\end{proposition}

Let $G$ be a finite permutation group acting on a finite set $\Omega$. If for any $\alpha,\beta\in\Omega$
there exists an element $g\in G$ such that $\beta=\alpha^g$, then $G$ is \textit{transitive} on $\Omega$.
Additionally, if such an element $g$ exists uniquely, then $G$ is \textit{regular} on $\Omega$.
A subset $\Delta\subseteq\Omega$ is a \textit{block} for $G$ if $\{\Delta^g|g\in G\}$ is
a partition of $\Omega$. A transitive permutation group $G$
on $\Omega$ is \textit{primitive} if the only blocks for $G$ are the singletons and the whole set $\Omega$,
and \textit{quasiprimitive} if  every nontrivial normal subgroup of $G$ is transitive.
It is well known that every primitive group is quasiprimitive.

\begin{proposition}\cite{Jones2010, Li2006}
Let $X$ be a  quasiprimitive permutation group of degree $n$. If $X$
contains a cyclic regular subgroup $H$. Then exactly one of the following holds:
\begin{enumerate}[\rm(a)]
  \item $\CC_p \cong H \leq X \leq \AGL(1, p)$, where $n=p$ is a prime.
  \item $X=\Alt_n$ where $n$ is odd, or $X=\Sym_n$, where $n \geq 4$;
  \item $\PGL(d, q) \leq X \leq \PGAL(d, q)$, $H$ is a Singer subgroup of $X$,
  and $n=(q^d-1)/(q-1)$.
  \item $X=\PGAL(2,8)$, $H=\langle s\sigma\rangle \cong \CC_9$,
  where $\langle s\rangle$ is a Singer subgroup of $X$ and $\sigma\in X \setminus \PSL(2,8)$ with $|\sigma|=3$.
  \item[\rm(e)] $(X, n)=(\PSL(2,11), 11),\,(\Mat_{11}, 11)$, or $(\Mat_{23}, 23)$.
\end{enumerate}
\end{proposition}

\begin{proposition}\cite[Theorem 1.5]{Li2006}
Let $X$ be a quasiprimitive permutation group on a finite set $\Omega$. If $X$
contains a dihedral regular subgroup $H$, then $X$ is 2-transitive,
and exactly one of the following holds for some $\omega \in \Omega$:
\begin{enumerate}[\rm(a)]
  \item $(X, H)=(\Alt_4, \DD_4), (\Sym_4, \DD_4), (\AGL(3,2), \DD_8),
  (\AGL(4,2), \DD_{16}), (\CC_2^4\rtimes \Alt_6, \DD_{16})$, or $(\CC_2^4\rtimes \Alt_7, \DD_{16})$.
  \item $(X, H, X_\omega)=(\Mat_{12}, \DD_{12}, \Mat_{11}),(\Mat_{22} .2, \DD_{22}, \PSL(3,4) .2)$,
  or $(\Mat_{24}, \DD_{24}, \Mat_{23})$.
  \item $(X, H, X_\omega)=(\Sym_{2 m}, \DD_{2 m}, \Sym_{2 m-1})$,
   or $(\Alt_{4 m}, \DD_{4 m}, \Alt_{4 m-1})$.
  \item $X=\PSL(2, p^e).O, H=\DD_{p^e+1}$, and
  $X_\omega \unrhd \CC_p^e \rtimes \CC_{\frac{p^e-1}{2}}.O$,
  where $p^e \equiv3\pmod 4$, and $O \leq \Out(\PSL(2, p^e)) \cong \CC_2 \times \CC_e$.
  \item $X=\PGL(2, p^e) \CC_f, H=\DD_{p^e+1}$, and $X_\omega=\CC_p^e \rtimes \CC_{p^e-1}$,
  where $p^e \equiv 1\pmod{4}$, and $f \mid e$.
\end{enumerate}
\end{proposition}

\begin{corollary}\label{primitive}
Let $X$ be a quasiprimitive permutation group with a stabilizer $H$ and a regular subgroup $K$.
If $X$ is solvable, then one of the the following holds true:
\begin{enumerate}[\rm(a)]
  \item If $K$ is cyclic, then either $(X, H, K)=(\Sym_4,\DD_6,\CC_4,)$, or $\CC_p \cong K \leq X \leq \AGL(1, p)$
  for some prime $p$.
  \item If $K$ is dihedral and $H$ is either cyclic or dihedral, then
  $(X, H, K)=(\Alt_4,  \CC_3,\DD_4)$ or $(\Sym_4,  \DD_6,\DD_4)$.
\end{enumerate}
\end{corollary}

The following three results on exact products related to dihedral groups will be useful.

\begin{proposition}\cite[Theorem 1]{Monakhov1974}\label{solvable}
If a finite group $X$ admits a factorization $X=HK$ of two subgroups $H$ and $K$,
each having a cyclic subgroup of index at most $2$, then $X$ is solvable.
\end{proposition}

\begin{proposition} \cite{DYL2023}\label{DM}
Let $X$ be an exact product of a dihedral subgroup $H=\langle x\rangle\rtimes\langle y\rangle\cong\DD_{2n}$
 and a cyclic subgroup $C=\langle z\rangle\cong\CC_m$, where $m,n\geq2$. Suppose that $M$ is
 the subgroup of maximal order in $X$ such that $\langle z\rangle\leq M\subseteq\langle x\rangle\langle z\rangle$.
 Then $M$, $M_X$ and $X/M_X$ fall into one of the following classes:
\end{proposition}
\begin{center}
\begin{threeparttable}[b]
\caption{Structure of $M$}\label{TAB}
\begin{tabular*}{70mm}[c]{|p{10mm}|p{15mm}|p{15mm}|p{15mm}|}
\toprule
 Case &$M$ & $M_X$  & $X/M_X$\\
  \hline
   1 & $\lg x\rg\lg z\rg$   & $\lg x\rg\lg z\rg$     &   $\CC_2$ \\
   2 &$\lg x^2\rg \lg z\rg$ & $\lg x^2\rg\lg z^2\rg$ &   $\DD_8$   \\
   3 &$\lg x^2\rg \lg z\rg$ & $\lg x^2\rg\lg z^3\rg$ &   $\Alt_4$   \\
   4 &$\lg x^4\rg \lg z\rg$ & $\lg x^4\rg\lg z^3\rg$ &   $\Sym_4$   \\
   5 &$\lg x^3\rg \lg z\rg$ & $\lg x^3\rg\lg z^4\rg$ &   $\Sym_4$   \\
\bottomrule
\end{tabular*}
\end{threeparttable}
\end{center}

\begin{proposition}\cite{DYL2023,ZD2016}\label{ac}
Let $X$ be an exact product of a dihedral subgroup $H=\lg x\rg\rtimes\lg y\rg\cong\DD_{2n}$ and a
 cyclic subgroup $\lg z\rg\cong\CC_m$, where $m,n\geq2$.
 If $n$ is odd, $\lg z\rg_X=1$ and $\lg x\rg\lg z\rg$ is a subgroup of $X$,
then $H\lhd X$.
\end{proposition}

\section{Characterizations}\label{sec:lems}
In this section, we give some structural characterizations of the exact products
of two non-abelian dihedral groups. To simplify the statements,
in the remainder of this paper, we use notation $X$ to denote an exact product
of two dihedral subgroups $H=\lg x\rg\rtimes\lg y\rg\cong\DD_{2n}$ and
$K=\lg z\rg\rtimes\lg w\rg\cong\DD_{2m}$, where $m,n\geq 3$.

\begin{lemma}\label{c1}
If $X$ has a maximal subgroup $G$ such that $H\leq G$ and $G\cap K\leq \lg z\rg$,
then
\begin{enumerate}[\rm(a)]
\item either $G=G_X=H\langle z\rangle$, $X=(H\lg z\rg)\rtimes\lg w\rg$, $X/G_X\cong\CC_2$, or
\item $G=H\langle z^2\rangle$, $G_X=\langle x^3,z^2\rangle$, $X=( K\lg x\rg)\rtimes\lg y\rg$ and
$X/G_X\cong\Sym_4$. In particular, $3|n$ and $2|m$.
\end{enumerate}
\end{lemma}
\begin{proof}
By hypothesis, we may assume $G\cap K=\lg z_1\rg$ for some $z_1\in\lg z\rg$.
Using Dedekind's modular law, we have $G=G\cap X=G\cap HK=H(G\cap K)=H\lg z_1\rg$.
Since
 \[
 \langle z_1\rangle\leq\bigcap_{k,l\in\ZZ} G^{z^kw^l}=\bigcap_{i,j,k,l\in\ZZ} G^{x^iy^jz^kw^l}=G_X,
 \]
we have  $G_X=G_X\cap G=G_X\cap H\langle z_1\rangle=(G_X\cap H)\lg z_1\rg$ and $G=HG_X$.

Consider now the quotient group $\ox:=X/G_X$. We have $\og=G/G_X=HG_X/G_X=\oh$
and $\ox=\oh\ok$, where $\ok=KG_X/G_X\cong K/\langle z_1\rangle.$
Since $ G$ is a maximal subgroup of $X$ and $\og_{\ox}=\overline{G_X}=1$,
the (faithful) transitive permutation representation of $\ox$ on the cosets $[\ox:\og]$
gives rise to a permutation group with a stabilizer $\og=\oh$ and a regular subgroup $\ok$.
Being homomorphic images of the dihedral subgroups $H$ and $K$, both $\oh$ and
$\ok$ are either dihedral or cyclic. Note that in the latter case, we have $|\oh|=1$ or $2$, and $|\ok|=2$.
By checking Corollary~\ref{primitive}, only two possible cases may occur: either
$\CC_p\cong\ok\leq\ox\leq \AGL(1, p)$
with $p=2$, or $(\ox,\oh,\ok)=(\Sym_4,\DD_6,\DD_4)$.
In the former case, we have $z_1=z$, and hence (a) holds true. In the latter case, we have $z_1=z^2$,
which gives (b), as required.
\end{proof}
\vskip 3mm

By reversing the roles of $H$ and $K$ in the above lemma, we obtain the following result.
\begin{lemma}\label{a1}
If $X$ has a maximal subgroup $G$ such that $K\leq G$ and $G\cap H\leq \lg x\rg$,
then
\begin{enumerate}[\rm(a)]
\item either $G=G_X=K\langle x\rangle$, $X=(K\lg x\rg)\rtimes\lg y\rg$, $X/G_X\cong\CC_2$, or
\item $G=K\langle x^2\rangle$, $G_X=\langle z^3,x^2\rangle$, $X=( H\lg z\rg)\rtimes\lg w\rg$ and
$X/G_X\cong\Sym_4$. In particular, $3|m$ and $2|n$.
\end{enumerate}
\end{lemma}

\begin{lemma}\label{m-n-odd}
If both $m$ and $n$ are odd, then $X=(H\lg z\rg)\rtimes\lg w\rg=(\lg x\rg K)\rtimes\lg y\rg$.
\end{lemma}
\begin{proof}
The hypothesis implies that the roles of $H$ and $K$ in $X$ are symmetric, so
 it suffices to show $X=(H\lg z\rg)\rtimes\lg w\rg$.

Let $G$ be a maximal subgroup $G$ containing $H$.
Then $G=G\cap X=G\cap HK=H(G\cap K)$.
If  $G\cap K\leq\lg z\rg$, then the result follows from Lemma~\ref{c1} (note that the case~(b)
cannot appear since both $m$ and $n$ are odd).

Suppose now $G\cap K\not\leq\langle z\rangle$. Then $G\cap K=\langle z_1,w_1\rangle$
for some $z_1\in\langle z\rangle$ and $w_1\in K\backslash\langle z\rangle$.
We may assume $w_1=w$, and so $G=H\langle z_1,w\rangle$. As before, we have
\[
\langle z_1\rangle\leq\bigcap_{k,l\in\ZZ} G^{z^kw^l}=\bigcap_{i,j,k,l\in\ZZ}G^{x^iy^jz^kw^l}=G_X,
\]
and the quotient group $\ox:=X/G_X$ is a primitive permutation group on $[\ox,\og]$ with
a stabilizer $\og$. Since $\langle \bar z\rangle$ is cyclic and transitive, it is
a regular subgroup of $\ox$. Now Corollary~\ref{primitive}(a) is applicable, we have
\[
\CC_p\cong\langle \bar z\rangle\leq\bar X\leq\AGL(1,p)\quad \text{for some odd prime $p$}.
\]
Thus, $z_1=z^p$ and the stabilizer $\og$ is cyclic. Since $\oh$ is a cyclic homomorphic
image of the dihedral group $H$ and $\oh\leq\og$,  we have
$\oh=\bar 1$ or $\oh=\langle\bar y\rangle\cong\CC_2$. In
either case, $x\in G_X$.

If $\oh=1$, then  $G_X=H\langle z^p\rangle$, and so $\ox=\langle\bar z^p\rangle\rtimes\langle\bar w\rangle$.
Therefore, $X=(H\langle z\rangle)\rtimes\langle w\rangle.$ If $\oh=\lg\bar y\rg$, then
$G_X=\lg x\rg\lg z^p\rg$ and $\ox=\langle\bar z,\bar w\rangle\rtimes \langle\bar y\rangle.$
Note that $\lg\bar z\rg\ch \lg \bar z,\bar w\rg\lhd \ox$, we have
$\lg\bar z\rg\lhd\ox$. Therefore, $X=(H\lg z\rg)\rtimes\lg w\rg,$ as desired.
\end{proof}

\begin{corollary}\label{Normal}
If both $m$ and $n$ are odd, then $X=(\lg z\rg_X\rtimes H).K=(\lg x\rg_X\rtimes K).H$
\end{corollary}
\begin{proof}
By Lemma~\ref{m-n-odd}, $X=(H\lg z\rg)\rtimes\lg w\rg$ and  the maximal
 subgroup $G:=H\lg z\rg$ of $X$ is an exact product
of the dihedral subgroup $H=\lg x\rg\rtimes\lg y\rg\cong\DD_{2n}$ and the
 cyclic subgroup $\lg z\rg\cong\CC_m$.
Since both $m$ and $n$ are odd, by Lemma~\ref{DM} we get $H\lg z\rg=(\lg x\rg\lg z\rg)\rtimes\lg y\rg$,
 and so $X=(\lg x\rg\lg z\rg)\rtimes\lg y\rg)\rtimes\lg w\rg.$ It is clear that $\lg z\rg_X=\lg z\rg_G$.
Thus, the quotient group $\og:=G/\lg z\rg_X$
contains a dihedral subgroup $\oh=\lg\bar x,\bar y\rg$ with a core-free cyclic complement
 $\lg\bar z\rg$. Since $\lg\bar x\rg\lg\bar z\rg\leq\og$,
by Proposition~\ref{ac}, we get $\oh\lhd \bar G$, and so $\og=\oh\rtimes\lg\bar z\rg$.
Since $\bar z$ has odd order,  $\bar g\in H$ for all $\bar g\in\og$ with $\bar g^2=1$,
so $\oh=\langle\bar g\in\og|\bar g^2=1\rg$. Thus $\oh\ch\og\lhd\ox$, and so $\oh\lhd\ox$.
Therefore, $X=(\lg z\rg_X\rtimes H).K$, as claimed.
\end{proof}

\section{Classification}\label{sec:class}
In this section, we utilize the results  from the previous section to provide a proof of Theorem~\ref{main}.
\begin{proof}[Proof of Theorem \ref{main}]
By Corollary~\ref{Normal}, we have $X=(\lg z\rg_X\rtimes H).K=(\lg x\rg_X\rtimes K).H$.
Therefore, $X$  contains the following two chains of
subnormal subgroups with cyclic factors:
\begin{equation}\label{m0}
1\leq\lg z\rg_X< \lg z\rg_X\rtimes \lg x\rg\leq\lg z\rg_X\rtimes H\leq (\lg z\rg_X\rtimes H).\lg z\rg<
\big((\lg z\rg_X\rtimes H).\lg z\rg\big) \rtimes \lg w\rg=X.
\end{equation}
and
\begin{equation}\label{n0}
1\leq\lg x\rg_X< \lg x\rg_X\rtimes \lg z\rg\leq\lg x\rg_X\rtimes K\leq (\lg x\rg_X\rtimes K).\lg x\rg<
\big((\lg x\rg_X\rtimes K).\lg z\rg\big) \rtimes \lg y\rg=X.
\end{equation}

Let $z_1:=z^{m_1}$ be a generator of  $\lg z\rg_X$, where $m_1$ is a positive divisor of $m$.
By $N/C$ theorem, $X/C_X(z_1)$ is abelian,  so its subgroup $H C_X(z_1)/C_X(z_1)$ is
also abelian. Note that $H C_X(z_1)/C_X(z_1)$ is a homomorphic image of the
dihedral group $H=\lg x\rg\rtimes\lg y\rg\cong\DD_{2n}$ with $n$ odd,
we must have $x\in C_X( z_1)$. Thus,
\begin{eqnarray}\label{m1}
z_1^x=z_1,\quad z_1^y=z_1^{u}
\end{eqnarray}
for some $u\in\ZZ_{m/m_1}$.

Consider $\ox:=X/\lg z_1\rg=\oh\rtimes\ok$. Since $\lg\bar x\rg\ch\oh\lhd\ox$, we have
 $\lg\bar x\rg\lhd \ox$. Thus, by the $N/C$ theorem again, $\bar z\leq C_{\ox}(\bar x)$.
It follows that the following relations hold in $\ox$:
\[
{\bar x}^{\bar z}={\bar x},\quad {\bar y}^{\bar z}={\bar x}^{R}{\bar y},\quad {\bar x}^{\bar w}={\bar x}^{S},
\quad{\bar y}^{\bar w}={\bar x}^{T}{\bar y},
\]
where $R,S,T\in\ZZ_n$. Consequently,   we obtain the following relations in $X$:
\begin{eqnarray}\label{m2}
x^z=xz_1^{e},\quad y^z=z_1^{a}x^{R}y,\quad x^w=x^{S}z_1^{b},
\quad y^w=yx^{T}z_1^{c}
\end{eqnarray}
or written in an equivalent form:
\begin{eqnarray}\label{m2-2}
[x,z]=z_1^{e},\quad [z,y]=x^{R}z_1^{a},\quad [x,w]=x^{S-1}z_1^{b}, \quad [y,w]=x^{T}z_1^{c},
\end{eqnarray}
where $e,a,b,c\in\ZZ_{m/m_1}$.

On the other hand, due to the formal symmetry between $H$ and $K$ in $X$, we can also
 consider the core $\lg x\rg_X$ of $\lg x\rg$ in $X$ and the corresponding quotient group $\ox:=X/\lg x\rg_X$.
Let $x_1:=x^{n_1}$ be a generator of $\lg x\rg_X$, where $n_1$ is a positive divisor
of $n$. Using a similar argument as before, it is straightforward to show that
$z\in C_X(x_1)$ and $\lg \bar z\rg\lhd\ox$.  Now we can deduce from \eqref{n0} and
\eqref{m2-2} that $R,T\equiv0\pmod{n_1}$ and $S\equiv1\pmod{n_1}$. Therefore, we may set
\[
R=n_1r, \quad S=1+n_1s \quad\text{and}\quad T=n_1t,\qquad\text{where $r,s,t\in\ZZ_{n/n_1}$}.
\]
Then the relations in \eqref{m2-2} can be expressed in the following symmetric form:
\begin{equation*}
[x,z]=z_1^{e},\quad [z,y]=x_1^{r}z_1^a, \quad [x,w]=x_1^{s}z_1^{b},\quad[y,w]=x_1^{t}z_1^c,
\end{equation*}
or written in an equivalent form as follows:
\begin{equation*}
x^z=xz_1^{e},  \quad y^z=x_1^{r}z_1^ay, \quad x^w=xx_1^{s}z_1^{b},\quad y^w=yx_1^{t}z_1^c.
\end{equation*}
Therefore, $X=\langle x,y,z,w\rangle$ has the following defining relations
\begin{equation}\label{m3-1-4}
\begin{aligned}
&x^n=y^2=z^m=w^2=1, x_1=x^{n_1}, z_1=z^{m_1},\\
&z_1^x=z_1, z_1^y=z_1^{u}, z_1^z=z_1, z_1^w=z_1^{-1},\\
&x^y=x^{-1}, x^z=xz_1^{e},  x^w=xx_1^{s}z_1^{b},\\
&y^z=x_1^{r}z_1^ay,   y^w=yx_1^{t}z_1^c, z^{w}=z^{-1}.
\end{aligned}
\end{equation}

In the subsequent analysis, we utilize Proposition~\ref{Cext} to exam the relations in
\eqref{m3-1-4} and  establish the numerical conditions on the parameters.

Since $y^2=1$ and $z_1^y=z_1^u$, we have $z_1=z_1^{y^2}=(z_1^u)^y=(z_1^y)^u=z_1^{u^2}$, and so
\begin{eqnarray}\label{cond1}
u^2\equiv 1\pmod {m/m_1}.
\end{eqnarray}
Moreover, from $x^y=x^{-1}$, $z_1=z^{m_1}$, $x^z=xz_1^e$ and $y^z=x_1^{r}z_1^ay$ we derive that
\begin{align*}
1&=x^n=(x^n)^z=(x^z)^n=x^nz_1^{ne};\\
1&=y^2=(y^2)^z=(y^z)^2=(x^{r}z_1^ay)^2=z_1^{a(u+1)};\\
x&=x^{z_1}=x^{z^{m_1}}=(xz_1^e)^{z^{m_1-1}}=x^{z^{m_1-1}}z_1^{e}=xz_1^{m_1e};\\
z_1^{u-1}y&=y^{z_1}=y^{z^{m_1}}=(x_1^{r}z_1^ay)^{z^{m_1-1}}=x_1^{r}z_1^ay^{z^{m_1-1}}=x_1^{m_1r}z_1^{m_1a}y;\\
x^{-1}z_1^{-e}&=(x^{-1})^z=(x^y)^z=(x^z)^{y^z}=(xz_1^e)^{x_1^{r}z_1^ay}=(xz_1^e)^y=x^{-1}z_1^{ue},
\end{align*}
which imply the following congruences
\begin{align}
&ne\equiv0\pmod{m/m_1},\label{cond2}\\
&a(u+1)\equiv0\pmod{m/m_1},\label{cond3}\\
&m_1e\equiv0\pmod{m/m_1},\label{cond4}\\
&m_1r\equiv0\pmod{n/n_1},\label{cond5}\\
&u\equiv 1+m_1a\pmod{m/m_1},\label{cond6}\\
&(u+1)e\equiv0\pmod{m/m_1}.\label{cond7}
\end{align}
Combining equations \eqref{cond4}, \eqref{cond6} with \eqref{cond7}, we obtain
\[
2e\equiv(u+1)e+(u-1)e\equiv (1-u)e\stackrel{\eqref{cond6}}\equiv (u+1)e+m_1ae\stackrel{\eqref{cond4}}\equiv0\pmod{m/m_1}.
\]
 Since $m$ is odd, we deduce that $e\equiv0\pmod{m/m_1}$, which implies $[x,z]=1$. Consequently,
 the congruences \eqref{cond1}--\eqref{cond7} are reduced to
\begin{align}\label{condA}
a(u+1)\equiv0,\quad u\equiv 1+m_1a\pmod{m/m_1}\quad\text{and}\quad m_1r\equiv0\pmod{n/n_1}.
\end{align}

Furthermore, from $x^w=xx_1^{s}z_1^b$ and $y^w=yx^{t}z_1^c$ we infer that
\begin{align*}
1&=x^n=(x^n)^w=(x^w)^n=(xx_1^{s}z_1^b)^n=z_1^{nb};\\
1&=y^2=(y^2)^w=(y^w)^2=(yx^{t}z_1^c)^2=z_1^{c(u+1)};\\
x&=x^{w^2}=(x^{1+n_1s}z_1^b)^w=(x^w)^{1+n_1s}z_1^{-b}\\
&=(xx_1^{s}z_1^b)^{1+n_1s}z_1^{-b}=x^{1+n_1s(2+n_1s)}z_1^{n_1sb};\\
x^{-1}x_1^{-s}z_1^{-b}&=(x^{-1})^w=(x^y)^{w}=(x^w)^{y^w}=(xx_1^{s}z_1^b)^{yx_1^tz_1^c}\\
&=(x^{-1}x_1^{-s}z_1^{bu})^{x_1^tz_1^c}=x^{-1}x_1^{-s}z_1^{bu};\\
yx_1^{t}z_1^c&=y^w=(y^w)^{z^wz}=(y^z)^{wz}=(x_1^{r}z_1^ay)^{wz}\\
&=\big((x_1^{1+sn_1}z_1^{bn_1})^{r}z_1^{-a}yx_1^{t}z_1^c\big)^z
=x_1^{r(1+sn_1)}z_1^{bn_1r-a}y^zx_1^{t}z_1^c\\
&=x_1^{r(1+sn_1)}z_1^{bn_1r-a}(x_1^{r}z_1^ay)x_1^{t}z_1^c=yx_1^{t-r(2+sn_1)}z_1^{ubn_1r+c}.
\end{align*}
Thus, we obtain
\begin{equation}\label{condB}
\begin{aligned}
&nb\equiv0\pmod{m/m_1},\\
&c(u+1)\equiv0\pmod{m/m_1},\\
&s(2+n_1s)\equiv0\pmod{n/n_1},\\
&n_1sb\equiv0\pmod{m/m_1},\\
&b(u+1)\equiv0\pmod{m/m_1},\\
&r(2+sn_1)\equiv0\pmod{n/n_1},\\
&n_1rb\equiv0\pmod{m/m_1}.
\end{aligned}
\end{equation}

In addition, for any integer $k$, we have
\begin{align*}
(x^k)^w=(xx_1^{s}z_1^b)^k=x^kx_1^{sk}z_1^{bk}\quad\text{and}\quad (z^k)^y&=(x_1^rzz_1^a)^k=x_1^{rk}z^kz_1^{ak}.
\end{align*}
Observe that $[x,z]=1$, $x^y=x^{-1}$ and $z^w=z^{-1}$. Using the assumption
$\lg x\rg_X=\lg x^{n_1}\rg$ and $\lg z\rg_X=\lg z^{m_1}\rg$
we infer from the above equations that
\begin{eqnarray}\label{condC}
\begin{aligned}
&rk\equiv 0\pmod {n/n_1}\iff k\equiv0\pmod{m_1},\\
&bk\equiv 0\pmod {m/m_1}\iff k\equiv0\pmod{n_1}.
\end{aligned}
\end{eqnarray}

Finally, we see from \eqref{condA} that $u=1+m_1a\pmod{m/m_1}$. Therefore, upon substitution, the congruences in \eqref{condA}, \eqref{condB} and \eqref{condC} are reduced to the stated form given in the theorem.
By Proposition~\ref{Cext}, the group $X$ defined by the stated presentation and numerical conditions is well-defined,
 it contains two dihedral subgroups $H=\lg x,y\rg\cong\DD_{2n}$ and $K=\lg z,w\rg\cong\DD_{2m}$ with $H\cap K=\{1_X\}$. In particular,  it is straightforward to verify that $\lg x\rg_X=\lg x^{n_1}\rg$ and $\lg z\rg_X=\lg z^{m_1}\rg$, as required.
\end{proof}

\end{document}